\documentclass[12pt, a4paper, reqno]{amsart}
\usepackage{graphicx}
\usepackage{amssymb, amsthm, xcolor}
\usepackage[centering, heightrounded]{geometry}
\usepackage[small,bf]{caption}
\usepackage[subrefformat=parens]{subcaption}
\numberwithin{equation}{section}

\newtheorem{theorem}{Theorem}[section]
\newtheorem{lemma}[theorem]{Lemma}

\theoremstyle{definition}
\newtheorem{definition}[theorem]{Definition}
\theoremstyle{remark}

\renewcommand{\l}{\mathopen{}\mathclose\bgroup\left}
\renewcommand{\r}{\aftergroup\egroup\right}

\def\norm[#1]{\left\Vert #1 \right\Vert}
\def\abs[#1]{\left\vert #1 \right\vert}

\def\tbra[#1,#2]{\left\langle #1 , #2\right\rangle} %

\title[Rearrangement inequalities]{Rearrangement inequalities of the one-dimensional maximal functions associated with general measures }

\author[X. NIE]{Xudong Nie}
\address{Department of Mathematics
and Physics, Shijiazhuang Tiedao University, Shijiazhuang 050043, P.R. China}
\email{nxd2016@stdu.edu.cn}

\author[D. Wu]{Di Wu}
\address{School of Science, Zhejiang University of Science and Technology, Hangzhou 310023, P.R. China}
\email{wudi08@mails.ucas.ac.cn}

\author[P. wang]{Panwang Wang}
\address{Department of Mathematics
and Physics, Shijiazhuang Tiedao University, Shijiazhuang 050043, P.R. China}
\email{wangpw@stdu.edu.cn}

\begin{document}

\setcounter{tocdepth}{1}

\begin{abstract}
We prove a rearrangement inequality for the  uncentered Hardy-Littlewood maximal function $M_{\mu}$ associate to general measure $\mu$ on $\mathbb{R}$.
This inequality is analogous to the Stein's result
 $cf^{**}(t)\leq(Mf)^{*}(t)\leq C f^{**}(t)$, where $f^*$ is the symmetric decreasing rearrangement function of $f$ and $f^{**}(t)=\int_0^tf^*(x)dx$.
 Moreover, we compute the best constant of $M_{\mu}$ on
 $L^{p,\infty}(\mathbb{R},d\mu)$.
\end{abstract}


\maketitle

\section{Introduction}
The decreasing rearrangement of a real-valued measurable function $f$
is the unique radial monotone decreasing function $f^*$ on $\mathbb{R}^n$
that is upper-continues and equimeasurable with $|f|$.
Usually,  equimeasurable rearrangements are used in analysis and mathematical
physics to study extremal problems.
We can reduce extremal problems involving functions on higher-dimentional
spaces to problems for functions of a single variable.
Moveover, the norms of $f$ and $f^*$ are equivalence in many commonly-used
function spaces, including the Lebesgue spaces $L^p$, the Lorentz spaces $L^{p,q}$
and the Orlicz spaces $L^{\phi}$,
making the rearrangement a very useful tool.

Rearrangement function was firstly introduced by Hardy, Littlewood, and P$\acute{o}$lya.
In the last chapter of the book \cite{hardy1},
they introduced two inequalities, the Hardy-littlewood inequality
\begin{equation}\label{Sec1.2}
\int_{\mathbb{R}}f(x)g(x)dx\leq \int_{\mathbb{R}}f^{*}(x)g^{*}(x)dx,
\end{equation}
and the Riesz rearrangement inequality
\begin{equation}\label{In001}
\int_{\mathbb{R}}f*g(x)h(x)dx\leq
\int_{\mathbb{R}}f^**g^*(x)h^*(x)dx,
\end{equation}
where $f*g(x)=\int_{\mathbb{R}}f(x-y)g(y)dy$.
The Hardy-Littlewood inquality (\ref{Sec1.2}) was generalized to $\mathbb{R}^n$ by Bennett in
\cite{B2}.
 The inequality (\ref{In001}) was generalized  to $\mathbb{R}^n$ by Sobolev in \cite{Sob}, and the inequality is
also called by Riesz-Sobolev inequality.
In \cite{R1}, Brascamp, Lieb and Luttinger found a generalization
of Riesz-Sobolev inequality to product of more than three positive functions:
\begin{equation}\label{wuyong01}
\int_{(\mathbb{R}^n)^p}
\prod_{j=1}^{k}f_j\left(\sum_{m=1}^{p}a_{jm}x_m\right)dx\leq
\int_{(\mathbb{R}^n)^p}
\prod_{j=1}^{k}f^*_j\left(\sum_{m=1}^{p}a_{jm}x_m\right)dx,
\end{equation}
where $a_{jm}$ are positive constants, $dx=dx_1\ldots dx_p$
and each $x_i\in\mathbb{R}^n$.
Further generalization of (\ref{wuyong01}) were given by
M.Christ \cite{C.001} and R.E.Pfiefer \cite{Pfi}.
Moreover, C.Draghici \cite{Dra} studied the rearrangement inequalities on other
metric spaces.

Inspired by  (\ref{In001}), we  studied the rearrangement inequality for
maximal convolution operator.
Let $\Phi$ be a nonnegative symmetric  decreasing function with $\|\Phi\|_{L^1}=1$.
Define $M_{\Phi}f(\cdot)=\sup_{|y-\cdot|\leq t}\Phi_t*g(y)$, where
$\Phi_t(\cdot)=\frac{1}{t^n}\Phi(\frac{\cdot}{t})$.
It is a nature question whether
 the analogy Riesz-Sobolev   inequality
\begin{equation}\label{in 002}
\int_{\mathbb{R}^n}M_{\Phi}f(x)g(x)dx
\leq
\int_{\mathbb{R}^n}M_{\Phi}f^*(x)g^*(x)dx
\end{equation}
holds.
When $n=1$ and $\Phi=\chi_{[-\frac12,\frac12]}$, we have proved
the inequality (\ref{in 002}) hold in \cite{Nie}.
In this sitiuation, the function $M_{\Phi}f$ become the uncentered
Hardy-Littlewood maximal function of $f$.

In this paper,  we will consider the uncentered maximal operator associated to general
positive Borel measure.
Suppose $\mu$ is a positive Borel measure.
Let $f$ be a $\mu$-measurable nonnegative function.
The uncentered  Hardy-Littlewood maximal function of $f$ corresponding to the
measure $\mu$  is defined by
$$
M_{\mu}f(x)=\sup_{x\in B,\mu(B)\neq 0}\frac{1}{\mu(B)}
\int_{ B}f(y)d\mu(y),
$$
where $B$ is the Euclidean ball.
If $\mu$ is the Lebesgue measure, we denote $M_{\mu}$ by $M$ for short.
When $n=1$,
Grafakos \cite{A2} showed that the   uncertered maximal operator
corresponding to the measure $\mu$
maps a function from $L^p$ to $L^p$ with upper bound $C_p$, where
$C_p$ is the unique positive of the solution
$$
(p-1)x^p-px^{p-1}-1=0.
$$
Moveover, Grafakos \cite{A3} proved that $C_p$ is sharp
when $\mu$ is  the Lebesgue measure.
In this paper, we will compute the best weak-$L^p$ constant
for maximal operators associated to general measures.

Another motivation for our study is   the arrangement inequalities for maximal functions given by Stein.
Denote $f^{**}(t)=\frac1t\int_{0}^{t}f^{*}(s)ds$.
Stein \cite{B2} showed  the  rearrangement inequality
$$
cf^{**}(t)\leq(M_{\mu}f)^*(t)\leq C f^{**}(t)
$$
holds for each $t>0$, where
$c$ and $C$ are positive constants independent of $f$ and $t$.
Pick \cite{PIC} generalized this result to the fractional maximal operator
$$
(M_{\gamma})f(x)=\sup_{x\in B}|B|^{\frac{\gamma}{n}-1}\int_{B}|f(y)|dy,
$$
and showed that
$$
c\sup_{t\leq\tau\leq\infty}\tau^{\gamma/n}f^{**}(\tau)
\leq(M_{\gamma}f)^*(t)\leq C\sup_{t\leq \tau\leq\infty}\tau^{\gamma/n}f^{**}(\tau)
$$
holds for $0<\gamma\leq n$.

Our main result is as following.
\begin{theorem}\label{Th1}
Suppose $f$ is a positive $\mu$-measurable function on $\mathbb{R}$ satisfying
(\ref{Sec1.1}). Then
we have
\begin{eqnarray}\label{main equality 1}
\left(M_{\mu}f\right)^*(x)\leq Mf^*(x)
\end{eqnarray}
for all $x\in \mathbb{R}$.
\end{theorem}
As an application of Theorem \ref{Th1},
we give the best constant of $M_{\mu}$ from
$L^{p,\infty}(\mathbb{R},d\mu)$ to itself.
\begin{theorem}\label{TH2}
Let $C_p$ be the  unique positive of the solution
$$
(p-1)x^p-px^{p-1}-1=0.
$$
Then we have
$$
\left\|M_{\mu}f\right\|_{L^{p,\infty}(\mathbb{R},d\mu)}
\leq C_p\left\|f\right\|_{L^{p,\infty}(\mathbb{R},d\mu)}.
$$
\end{theorem}

\section{Preliminaries and Lemmas}
Let $\mu$ be a positive Borel measure on $\mathbb{R}^n$
and $f$ be a $\mu$-measurable function
which satisfies the condition
\begin{equation}\label{Sec1.1}
\mu\left(\{|f|>t\}\right)<\infty
\end{equation}
 for each $t>0$, where we write
$$\{|f|>t\}=\{x\in\mathbb{R}^n:|f(x)|>t\}.$$
We define the  distribution function of $f$ by
$$
\lambda_{f}(t)=\mu(\{|f|>t\}).
$$
\begin{definition}
Suppose that $\mu$ is a positive Borel measure on $\mathbb{R}^n$
and $f$ is a $\mu$-measurable function
which satisfies (\ref{Sec1.1}).
The  symmetric decreasing rearrangement function of $f$ on
$\mathbb{R}^n$ is defined to be the function $f^*$ such that
$$
f^*(x)=\inf\{t:\lambda_{f}(t)\leq \mu(B(0,|x|))\}.
$$
\end{definition}

It follows that $f^*$  satisfies

(i) $\mu\left(\{|f|>t\}\right)=
\left|
\{f^*>t\}
\right|$
for each $t>0$, where  $\left|\cdot\right|$ denotes the Lebesgue measure,

(ii) $f^*(x)=f^*(y)$ when $|x|=|y|$,

(iii) $f^*(y)\leq f^*(x)$ when $|x|\leq |y|$.

Suppose $E$ is a finite $\mu$-measurable set on $\mathbb{R}$.
the symmetric rearrangement of the set $E$ is defined
to the open interval  centered at the origin
whose length is
the measure of $E$ and denoted by $E^*$.
Thus,
$$
E^*=\left(-\frac{\mu(E)}{2},\frac{\mu(E)}{2}\right).
$$
Suppose $f$ is a $\mu$-measurable function on $\mathbb{R}$ satisfying
(\ref{Sec1.1}).
Lieb \cite{Lieb Sharp constant} showed that
\begin{equation}
f^*(x)=\int_{0}^{\infty}\chi_{\{|f|>t\}}^*(x)dt,
\end{equation}
which is called the cake representation of $f^*$.

Let $E$  be a finite $\mu$-measurable set.
Taking $g=\chi_{E}$, (\ref{Sec1.2}) implies the following lemma.

\begin{lemma}\label{2019031301}
(see \cite{B1})
Suppose $f$ is a positive $\mu$-measurable function on $\mathbb{R}$ satisfying (\ref{Sec1.1}).
Let $E$ be a finite $\mu$-measurable set in $\mathbb{R}$.  Then we have
\begin{equation}
\int_{E}f(x)d\mu(x)\leq \int_{E^*}f^*(x)dx.
\end{equation}
\end{lemma}

Our result is  depend on the special geometry structure
 of $\mathbb{R}$ which is provided in \cite{G2007}.
We show this result here.

\begin{lemma}\label{Th2.2}
Suppose $\mathcal{F}$ is  a finite set of open intervals
in $\mathbb{R}$. Then there exist two subfamilies
$\mathcal{F}_1$ and $\mathcal{F}_2$ such that
$$
I\cap J=\emptyset \,\,\,\,\,\,\mbox{if}\,\,I,J\in \mathcal{F}_i
\,\,\mbox{and}\,\, I\neq J
$$
and
$$
\left(\cup_{I_1\in \mathcal{F}_1}\right)I_1\cup \left(\cup_{I_2\in \mathcal{F}_2}\right)I_2
=\cup_{I\in \mathcal{F}}I.
$$
\end{lemma}

A positive function  $f$ defined on $\mathbb{R}$ is called distributional
symmetric if
$$
|\{x:x>0\,\,\mbox{and}\,\,f(x)>t\}|=|\{x:x<0\,\,\mbox{and}\,\,f(x)>t\}|
$$
for each $t>0$.

\begin{lemma}\label{Lm01.1}
Suppose $a,b\geq 0$ and $f$ is a distributional symmetric function defined on $\mathbb{R}$.
Then we have
$$
\int_{-a}^{b}f(x)dx\leq \int_{-a}^{b}f^*(x)dx.
$$
\end{lemma}
\begin{proof}
Denote $f_a=f\chi_{[-a,0]}$ and $f_b=f\chi_{[0,b]}$.
Obviously, we have $|\{f_a\geq t\}|\leq a$.
Note that $f$ is a distributional symmetric function.
It follows that $|\{f_a\geq t\}|\leq \frac12|\{f^*>t\}|$.
Hence,
\begin{equation}\label{yili21}
|\{f_a\geq t\}|\leq \min\left\{a,\frac12|\{f^*>t\}|\right\}.
\end{equation}
Because $f^*$ is a symmetric decreasing function,
it can be deduced
\begin{equation}\label{yinli22}
|\{x:-a<x<0\,\,\mbox{and}\,\,f^*(x)>t\}|=\left\{
\begin{array}{ll}
a,\quad & \mbox{{\rm if}\quad $t<
f^*(a)$,}\\
\frac12|\{f^*>t\}|,\quad& \mbox{{\rm if} $t\geq f^*(a)$}.\\
\end{array}
\right.
\end{equation}
(\ref{yili21}) and (\ref{yinli22}) imply
\begin{equation}\label{yinli23}
|\{f_a\geq t\}|\leq |\{x:-a<x<0\,\,\mbox{and}\,\,f^*(x)>t\}|.
\end{equation}

Using Fubini lemma,
\begin{eqnarray}\label{eq 2.1}
\int_{-a}^0f(x)dx&=&\int_{\mathbb{R}}f_a(x)dx\nonumber\\
&=&\int_{\mathbb{R}}\int_{0}^{f_a(x)}dsdx\nonumber \\
&=&\int_{0}^{\infty}
|\{f_a>t\}|dt
\end{eqnarray}
By (\ref{yinli23}) and (\ref{eq 2.1}), it follows
\begin{equation}\label{eq 2.2.1}
\int_{-a}^0f(x)dx\leq \int_{-a}^{0}f^*(x)dx.
\end{equation}

Similarly, we can obtain
\begin{equation}\label{eq 2.2}
\int_{0}^bf(x)dx\leq \int_{0}^{b}f^*(x)dx.
\end{equation}
Using (\ref{eq 2.2.1}) and (\ref{eq 2.2}), it can be deduced
$$
\int_{-a}^bf(x)dx\leq \int_{-a}^{b}f^*(x)dx.
$$
The proof is completed.
\end{proof}

\section{The proof of the two Theorems}
\subsection{Theorem \ref{Th1}}

By the layer cake representation for
$\left(M_{\mu}f\right)^*$ and $Mf^*$,
inequality (\ref{main equality 1}) is equivalent to
\begin{equation}
\int_0^{\infty}\chi_{\{x:\left(M_{\mu}f\right)^*(x)>\lambda\}}(x)d\lambda
\leq \int_0^{\infty}\chi_{\{x:Mf^*(x)>\lambda\}}(x)d\lambda.
\end{equation}
Since $|\{x:\left(M_{\mu}f\right)^*(x)>\lambda\}|=
\mu\left(\{x:M_{\mu}f(x)>\lambda\}\right)$,
it is suffice to prove
\begin{equation}
\mu\left(\{x:M_{\mu}f(x)>\lambda\}\right)\leq
|\{x:Mf^*(x)>\lambda\}|
\end{equation}
for all $\lambda>0$.

Denote $E_{\lambda}^{\mu}(f)=
\{x\in\mathbb{R}:M_{\mu}f(x)>\lambda\}$.
When $\mu$ is the Lebesgue measure,
we write $E_{\lambda}(f)$ for simplicity.
Choosing  $x\in E_{\lambda}^{\mu}(f)$,
there is  an open interval $I_x$ such that
\begin{equation}\label{th in1}
\frac{1}{\mu(I_x)}\int_{I_x}fd\mu>\lambda.
\end{equation}
It follows that
$$E_{\lambda}^{\mu}(f)=\bigcup_{x\in E_{\lambda}^{\mu}(f)}I_x.$$
By Lindel$\ddot{o}$f's theorem, we can find
 a countable subcollection $I_j$,
$j=1,2,\ldots$, such that
$$
E_{\lambda}^{\mu}(f)=\bigcup_{j=1}^{\infty}I_j.
$$

Denote
$
F(N)=\bigcup_{j=1}^{N}I_{j}.
$
Using Lemma \ref{Th2.2}, there exists  two subcollections
$\mathcal{A}_1(N)$, $\mathcal{A}_2(N)$ of $\{I_1,I_2,\cdots,I_N\}$
so that
\begin{equation}\label{M3.1}
I\cap J=\emptyset \,\,\,\,\,\,\mbox{if}\,\,I,J\in \mathcal{A}_i(N)
\,\,\mbox{and}\,\, I\neq J
\end{equation}
and
\begin{equation}\label{M3.2}
F(N)=\bigcup_{i=1}^{2}\left(\bigcup_{I\in \mathcal{A}_i(N)}I\right).
\end{equation}

Set $F_i=\bigcup_{I\in \mathcal{A}_i}I$, $i=1,2$.
By  (\ref{th in1}) and (\ref{M3.1}), it can be deduced that
\begin{equation}\label{equ 111}
\mu(F_i)=\sum_{I\in{\mathcal{A}_i}}\mu(I)
<\frac{1}{\lambda}\sum_{I\in\mathcal{A}_i}\int_{I}fd\mu
=\frac{1}{\lambda}\int_{F_i}fd\mu.
\end{equation}

We set $a=\frac12\mu(F_1\cup F_2)$,
$b=\frac12\mu(F_1\cap F_2)$, $b_{12}=\frac12\mu(F_1\backslash F_2)$
and $b_{21}=\frac12\mu(F_2\backslash F_1)$.
Obviously, we have
$$
\mu(F_1\cup F_2)=\mu(F_1\cap F_2)+\mu(F_1\backslash F_2)+\mu(F_2\backslash F_1)
$$
$$
\mu(F_1)=\mu(F_1\cap F_2)+\mu(F_1\backslash F_2)
$$
and
$$
\mu(F_2)=\mu(F_1\cap F_2)+\mu(F_2\backslash F_1).
$$
Hence
\begin{equation}
a=b+b_{12}+b_{21},
\end{equation}
\begin{equation}\label{M3.1.6}
\mu(F_1)=2(b+b_{12}),
\end{equation}
and
\begin{equation}\label{M3.1.7}
\mu(F_2)=2(b+b_{21}).
\end{equation}

Set
$$
f_1=f\chi_{F_1\cap F_2},\ \ \ f_{12}=f\chi_{F_1\backslash F_2}\ \ \ \mbox{and}\ \ \
f_{21}=f\chi_{F_2\backslash F_1}.
$$
Rearranging $f_{1}$,
$f_{12}$ and $f_{21}$,
we obtain corresponding symmetric-decreasing functions
$f_1^*$, $f_{12}^*$ and $f_{21}^*$  respectively.
It follows from Lemma \ref{2019031301} that
\begin{equation}\label{M3.1.62}
\int_{F_1\cap F_2}f(x)d\mu(x)\leq
\int_{\left(F_1\cap F_2\right)^*}f_1^*(x)dx,
\end{equation}
\begin{equation}\label{M3.1.63}
\int_{F_1\backslash F_2}f(x)d\mu(x)\leq
\int_{\left(F_1\backslash F_2\right)^*}f_{12}^*(x)dx,
\end{equation}
and
\begin{equation}\label{M3.1.72}
\int_{F_2\backslash F_1}f(x)d\mu(x)\leq
\int_{\left(F_2\backslash F_1\right)^*}f_{21}^*(x)dx.
\end{equation}
Since $F_1=(F_1\cap F_2)\cup (F_1\backslash F_2)$
and $F_2=(F_2\cap F_1)\cup (F_2\backslash F_1)$,
using (\ref{equ 111}) we obtain
\begin{equation}\label{M3.1.61}
\mu(F_1)\leq\frac{1}{\lambda}
\left(\int_{F_1\cap F_2}f_1(x)dx
+\int_{F_1\backslash F_2}f_{12}(x)dx\right),
\end{equation}
and
\begin{equation}\label{M3.1.73}
\mu(F_2)\leq\frac{1}{\lambda}
\left(\int_{F_1\cap F_2}f_1(x)dx
+\int_{F_2\backslash F_1}f_{21}(x)dx\right).
\end{equation}
Combining (\ref{M3.1.7}),
(\ref{M3.1.62}), (\ref{M3.1.63}) and (\ref{M3.1.61}), we  have
\begin{equation}\label{main 0001}
2(b+b_{12})<\frac{1}{\lambda}
\left(\int_{-b}^{b}f_1^*(x)dx
+\int_{-b_{12}}^{b_{12}}f_{12}^*(x)dx\right)
\end{equation}
Combining (\ref{M3.1.6}),
(\ref{M3.1.62}), (\ref{M3.1.72}) and (\ref{M3.1.73}), we have
\begin{equation}\label{main 0002}
2(b+b_{21})<\frac{1}{\lambda}
\left(\int_{-b}^{b}f_1^*(x)dx
+\int_{-b_{21}}^{b_{21}}f_{21}^*(x)dx\right)
.
\end{equation}
Since $f_1^*$, $f_{12}^*$ and $f_{21}^*$ are even functions, it can be deduced that
\begin{equation}\label{main 0001}
b+b_{12}<\frac{1}{\lambda}
\left(\int_{0}^{b}f_1^*(x)dx
+\int_{0}^{b_{12}}f_{12}^*(x)dx\right)
\end{equation}
and
\begin{equation}\label{main 0002}
b+b_{21}<\frac{1}{\lambda}
\left(\int_{-b}^{0}f_1^*(x)dx
+\int_{0}^{b_{21}}f_{21}^*(x)dx\right)
.
\end{equation}
Add both sides  of equations (\ref{main 0001}) and (\ref{main 0002}), we obtain
\begin{equation}\label{main 00020}
2b+b_{12}+b_{21}<\frac{1}{\lambda}
\left(\int_{-b}^{b}f_1^*(x)dx+\int_{0}^{b_{12}}f_{12}^*(x)dx
+\int_{0}^{b_{21}}f_{21}^*(x)dx\right).
\end{equation}

Define a new function by
\begin{equation*}
f^{\diamond}(x)=\left\{
\begin{array}{ll}
f_1^*(x) & 0\leq |x|<b\\
f_{12}^*(x-b)& b\leq |x|<b+b_{12}\\
f_{21}^*(x-b-b_{12})& b+b_{12}\leq |x|<b+b_{12}+b_{21}\\
0& others
\end{array}
\right.
\end{equation*}
Thus, (\ref{main 00020}) implies
\begin{equation}\label{M0001}
2b+b_{12}+b_{21}<\frac{1}{\lambda}
\left(\int_{-b}^{b+b_{12}+b_{21}}f^{\diamond}(x)dx\right).
\end{equation}
An argument similar to  the proof of (\ref{M0001}) shows that
\begin{equation}\label{M0002}
2b+b_{12}+b_{21}<\frac{1}{\lambda}
\left(\int_{-b-b_{12}-b_{21}}^{b}f^{\diamond}(x)dx\right).
\end{equation}

By the definition, $f^{\diamond}$ is a distribution symmetric  function and
\begin{equation}\label{xin}
\left(f^{\diamond}\right)^*\leq f^*.
\end{equation}
.
It can be deduced from (\ref{M0001}) (\ref{xin}) and Lemma \ref{Lm01.1} that
\begin{equation}\label{201905111}
2b+b_{12}+b_{21}
\leq \frac{1}{\lambda}
\left(\int_{-b}^{b+b_{12}+b_{21}}f^{*}(x)dx\right).
\end{equation}
On the other hand, by (\ref{M0002}) (\ref{xin}) and Lemma \ref{Lm01.1}, we have
\begin{equation}\label{201905112}
2b+b_{12}+b_{21}
\leq \frac{1}{\lambda}
\left(\int_{-(b+b_{12}+b_{21})}^{b}f^{*}(x)dx\right).
\end{equation}
Thus,
we can deduce from (\ref{201905111})
and
(\ref{201905112}) that
\begin{equation}
\frac{1}{2b+b_{12}+b_{21}}\int_{-b}^{b+b_{12}+b_{21}}f^*(x)dx>\lambda,
\end{equation}
and
\begin{equation}
\frac{1}{2b+b_{12}+b_{21}}\int_{-(b+b_{12}+b_{21})}^{b}f^*(x)dx>\lambda.
\end{equation}
Now it is easy to know
$$
(-b,b+b_{12}+b_{21})\subset
\{x:Mf^*(x)>\lambda\}
$$
and
$$
(-(b+b_{12}+b_{21}),b)\subset
\{x:Mf^*(x)>\lambda\}.
$$
Thus we have
\begin{equation}\label{2019051103}
\left(-(b+b_{12}+b_{21}),b+b_{12}+b_{21}\right)\subset
\{x:Mf^*(x)>\lambda\}.
\end{equation}
This implies that
$$
\mu\left(F^N\right)=\mu\left(F_1\cup F_2\right)=2(b+b_{12}+b_{21})\leq |\{x:Mf^*(x)>\lambda\}|.
$$
Since $F^N$ is an increasing sequence of $\mu$-measurable
sets and its union is $\{x:M_{\mu}f(x)>\lambda\}$,
the result follows by letting $N\rightarrow \infty$.
\subsection{Theorem \ref{TH2}}
Without of generality, we assume $\|f\|_{L^{p,\infty}(R,d\mu)}=1$.
Thus we have
\begin{equation*}\label{TH2 eq1}
\|f^*\|_{L^{p,\infty}(R,dm)}=
\|f\|_{L^{p,\infty}(R,d\mu)}=1,
\end{equation*}
since $f$ and $f^*$
are equimeasurable.
By the definition of $\|f\|_{L^{p,\infty}(R,d\mu)}$,
we have
\begin{equation}
\|f\|_{L^{p,\infty}(\mathbb{R},d\mu)}=
\sup_{t\in\mathbb{R}}|2t|^{\frac{1}{p}}f^*(t).
\end{equation}
It follows that
$f^*(t)\leq |2t|^{-\frac1p}$.
Denote $a(t)=|2t|^{-\frac1p}$.
Hence,
\begin{equation}\label{zuihou}
Mf^*(t)\leq Ma(t).
\end{equation}

Nextly, let us compute $Ma(t)$.
Fixing $t>0$,
observe that $a(t)$ is radial decreasing.
Hence, there exists $s>0$ such that
$$
Ma(t)=\sup_{a>0}\frac{1}{a+s}\int_{-s}^{t}a(u)du.
$$
Let $b(s)=\frac{1}{a+s}\int_{-s}^{t}a(u)du$.
Thus the derived function of $b(s)$ is
$$
\frac{db(s)}{ds}=-\frac{1}{(s+t)^2}\int_{-s}^{t}a(u)du
+\frac{a(s)}{s+t}.
$$
Then the maximum point $s_0$ of $b(s)$ satisfies
\begin{equation}\label{TH2 eq2}
\frac{1}{s_0+t}\int_{-s_0}^{t}a(u)du=a(s_0).
\end{equation}
On the other hand, we have
\begin{eqnarray}\label{TH2 eq3}
\int_{-s}^{t}a(u)du&=&\int_{-s}^0(-2u)^{-\frac1p}du
+\int_{0}^{t}(2u)^{-\frac1p}du\nonumber\\
&=&\frac{p}{p-1}2^{-\frac1p}\left(s^{1-\frac1p}+t^{1-\frac1p}\right)
\end{eqnarray}

Set $w=\frac{s_0}{t}$.
Combing (\ref{TH2 eq2}) and (\ref{TH2 eq3}),
$w$ satisfies the equation
$$
\frac{1}{p-1}w^{1-\frac1p}-w^{-\frac1p}+\frac{p}{p-1}=0.
$$
This tells us that $w$ is positive, unique
and equal to $c_p^{-p}$.
Since $w$ is unique, by ($\ref{TH2 eq1}$),
we get
$$
Ma(t)=\frac{1}{t+s_0}\int_{-s}^{t}a(u)du=|2s_0|^{-\frac1p}
=c_p|2t|^{-\frac1p}.
$$
Hence
$$
\|Ma\|_{L^{p,\infty},(\mathbb{R},dm)}\leq c_p.
$$
It can be deduced  from (\ref{zuihou}) that
$$
\|M_{\mu}f\|_{L^{p,\infty},(\mathbb{R},d\mu)}\|
\leq \|Mf^*\|_{L^{p,\infty},(\mathbb{R},d\mu)}\|
\leq c_p.
$$
The proof is completed.

\section*{Acknowledgments}

The research was supported by the
Hebei Province introduced overseas student support projects
 (Grant Nos.C20190365), NSF of Zhejiang Province of China (Grant
No. LQ18A010002) and National Natural Foundation of China (Grant
No.12001488).


\end{document}